\documentclass[]{amsart}
\usepackage[Symbol]{upgreek}

\usepackage{amssymb}
\usepackage{enotez}
\usepackage{mathrsfs}
\usepackage{accents}
\usepackage{listliketab}
\usepackage{stmaryrd}

\usepackage{hyperref}

\usepackage{etoolbox} 

\newcommand{\bbold}{\mathbb}

\def \O{\mathcal{O}}
\def\HLO{\Upl\Upo}

\def\R { {\bbold R} }
\def\Q { {\bbold Q} }
\def\Z { {\bbold Z} }

\def\T { {\bbold T} }

\def \I{\operatorname{I}}

\renewcommand\epsilon{\varepsilon}

\def \d{\operatorname{d}}

\def \<{\langle}
\def \>{\rangle}

\def \dv{\operatorname{dval}}

\def \hat {\widehat}

\def \supp {\operatorname{supp}}

\def \((  {(\!(}
\def \)) {)\!)}

\def \res{\operatorname{res}}

\def \k {{{\boldsymbol{k}}}}

\DeclareMathSymbol{\precequ}{\mathrel}{symbols}{"16}
\DeclareMathSymbol{\succequ}{\mathrel}{symbols}{"17}

\def \nasymp{\not\asymp}

\newcommand{\claim}[2][\!\!]{\medskip\noindent {\sc Claim #1:} {\it #2}\medskip}

\newcommand{\subcase}[2][\!\!]{\medskip\noindent {\sc Subcase #1:} {\it #2}\/}

\newtheorem{theorem}{Theorem}[section]
\newtheorem{lemma}[theorem]{Lemma}
\newtheorem{prop}[theorem]{Proposition}
\newtheorem{cor}[theorem]{Corollary}

\newtheorem*{theoremA}{Theorem A}
\newtheorem*{corB}{Corollary B}
\newtheorem*{corC}{Corollary C}

\theoremstyle{definition}

\theoremstyle{remark}
\newtheorem*{example}{Example}

\newtheorem{remarkNumbered}[theorem]{Remark}

\newcommand{\abs}[1]{\lvert#1\rvert}

\def \dv{\operatorname{dv}}

\let\oldi\i
\let\oldj\j
\renewcommand\i{\relax\ifmmode{\boldsymbol{i}}\else\oldi\fi}
\renewcommand\j{\relax\ifmmode{\boldsymbol{j}}\else\oldj\fi}

\renewcommand\leq{\leqslant}
\renewcommand\geq{\geqslant}

\renewcommand\le{\leq}
\renewcommand\ge{\geq}

\DeclareMathAlphabet{\mathbf}{OML}{cmm}{b}{it}

\DeclareFontFamily{U}{fsy}{}
\DeclareFontShape{U}{fsy}{m}{n}{<->s*[.9]psyr}{}
\DeclareSymbolFont{der@m}{U}{fsy}{m}{n}
\DeclareMathSymbol{\der}{\mathord}{der@m}{182}

\DeclareSymbolFont{der@m}{U}{fsy}{m}{n}
\DeclareMathSymbol{\derdelta}{\mathord}{der@m}{100}

\DeclareSymbolFont{imag@m}{OT1}{cmr}{m}{ui}
\DeclareMathSymbol{\imag}{\mathord}{imag@m}{105}

\DeclareFontFamily{OMS}{smallo}{}
\DeclareFontShape{OMS}{smallo}{m}{n}{<->s*[.65]cmsy10}{}
\DeclareSymbolFont{smallo@m}{OMS}{smallo}{m}{n}
\DeclareMathSymbol{\smallo}{\mathord}{smallo@m}{79}

\DeclareFontFamily{OMS}{largerdot}{}
\DeclareFontShape{OMS}{largerdot}{m}{n}{<->s*[.8]cmsy10}{}
\DeclareSymbolFont{largerdot@m}{OMS}{largerdot}{m}{n}
\DeclareMathSymbol{\largerdot}{\mathord}{largerdot@m}{15}

\DeclareMathSymbol{\llambda}{\mathord}{der@m}{108}
\DeclareMathSymbol{\rrho}{\mathord}{der@m}{114}

\def \Upg{\Upgamma}
\def \upl{\uplambda}
\def \Upl{\Uplambda}
\def \upo{\upomega}
\def \Upo{\Upomega}

\newcommand{\equationqed}[1]{\[\pushQED{\qed}#1 \qedhere\popQED\]\let\qed\relax}
\newcommand{\alignqed}[1]{\begin{align*}\pushQED{\qed} #1 \qedhere\popQED\end{align*}\let\qed\relax}

\def \No{\text{{\bf No}}}

\begin{document}
\title{Short Hardy Fields}

\author[Aschenbrenner]{Matthias Aschenbrenner}
\address{Kurt G\"odel Research Center for Mathematical Logic\\
Universit\"at Wien\\
1090 Wien\\ Austria}
\email{matthias.aschenbrenner@univie.ac.at}

\author[van den Dries]{Lou van den Dries}
\address{Department of Mathematics\\
University of Illinois at Urbana-Cham\-paign\\
Urbana, IL 61801\\
U.S.A.}
\email{vddries@illinois.edu}

\date{July 2025.}

\begin{abstract} Differentially algebraic Hardy field extensions of short Hardy fields  are short.
This is proved in the more general setting of $H$-fields. As an application we extend a theorem of
Rosenlicht (1981) by showing that each short asymptotic couple of Hardy type with small derivation  is isomorphic to the asymptotic couple of an analytic Hardy field.
\end{abstract}

\maketitle

\section*{Introduction} 

\noindent
An ordered set---here and below ``ordered'' means ``totally ordered''--- is said to be {\em short\/} if each ordered
subset of it has countable cofinality and countable coinitiality. {\em Example}\/: the real line. Shortness is a rather robust property, and \cite[Section~5]{AD} considers this property for Hardy fields (which are naturally ordered fields) and the ordered differential field $\T$ of transseries. In fact, $\T$ is short \cite[Corollary~5.20]{AD}. Left open in \cite{AD} is whether every differentially algebraic Hardy field extension of a short Hardy field is short. Here we give an affirmative answer. We actually prove something more general for $H$-fields. 

An {\em $H$-field\/}  is by definition an ordered field $H$ with a derivation $h\mapsto h'$ on it that interacts with the ordering as follows: for the constant field $C$ of $H$ and the convex
subring $\O:=\big\{h\in H:\text{$|h|\le c$ for some $c\in C$}\big\}$ of $H$ we have: \begin{enumerate}
\item[(H1)] for all $h\in H$, if $h>C$, then $h'>0$;
\item[(H2)] $\O=C+\smallo$, where $\smallo$ is the maximal ideal of the valuation ring $\O$ of $H$.
\end{enumerate} 
Hardy fields that contain $\R$ as a subfield are $H$-fields with constant field $\R$, as is $\T$. 

\begin{theoremA}\label{main}  If $E$ is a differentially algebraic $H$-field extension of a short $H$-field and the constant field of $E$ is short, then $E$ is short. 
\end{theoremA}

\noindent
This answers the above question from \cite{AD} for Hardy fields containing $\R$. Let $H$ be any short Hardy field and $E$ a differentially algebraic Hardy field extension of $H$. Then the Hardy field $H(\R)$ is short, by Lemma~\ref{hfs} below,
and $E(\R)$ is a differentially algebraic Hardy field extension of $H(\R)$, so $E(\R)$ and thus $E$ are short. This answers the question for all Hardy fields.

We use Theorem~A to realize short $H$-fields and short asymptotic couples in the realm of Hardy fields.
Some terminology:
A Hardy field is said to be {\it analytic}\/ if each element of it has an analytic representative~${(a,+\infty)\to\R}$ ($a\in\R$). An analytic Hardy field not contained in any strictly larger analytic Hardy field is called {\it maximal.}\/ (By Zorn,
each analytic Hardy field extends to a maximal one. Every maximal analytic Hardy field contains $\R$.) A valued differential field $K$ is said to have {\it small derivation}\/ if $\smallo'\subseteq \smallo$, where $\smallo$ is the maximal ideal of the valuation ring of $K$. 
  Hardy fields have small derivation, as has $\T$.
By \cite[Corollary~7.9]{AD}, $\T$ is isomorphic to an analytic Hardy field.
We improve this as follows:

\begin{corB}
Every short $H$-field  with small derivation and archimedean constant field embeds into every maximal analytic Hardy field.
\end{corB}

\noindent
In statements like these, embeddings (and isomorphisms as a special case) are valued field embeddings that respect the ordering and the derivation. 
As an application of Corollary~B and its proof we show that the $H$-field $\No(\omega_1)$ with the derivation from \cite{BM} embeds into every maximal analytic Hardy field.

\medskip\noindent
Let $H$ be an $H$-field, and let $h\mapsto vh\colon H^\times\to\Gamma=v(H^\times)$ be the valuation on $H$ with valuation ring $\mathcal O$ as above.
Then the logarithmic derivative map 
$$h\mapsto h^\dagger:=h'/h\ \colon\  H^\times\to H$$ 
descends to $\Gamma$:
there is a map $\psi\colon \Gamma^{\ne}=\Gamma\setminus\{0\}\to\Gamma$ such that $\psi(vh)=v(h^\dagger)$ for all $h\in H^\times$ with $vh\ne 0$, and such that for all $\alpha,\beta\in\Gamma^{\ne}$:
\begin{itemize}
\item[(A1)] $\alpha+\beta\ne 0\Rightarrow \psi(\alpha+\beta)\ge \min\big(\psi(\alpha),\psi(\beta)\big)$;
\item[(A2)] $\psi(k\alpha)=\psi(\alpha)$ for $k\in\Z\setminus\{0\}$;
\item[(A3)] $\alpha>0\Rightarrow\alpha+\psi(\alpha)>\psi(\beta)$.
\end{itemize}
The pair $(\Gamma,\psi)$ is called the {\it asymptotic couple}\/ of $H$. Any pair $(\Gamma,\psi)$
where $\Gamma$ is an ordered abelian group and $\psi\colon\Gamma^{\ne}\to\Gamma$ satisfies (A1)--(A3) for all $\alpha,\beta\in\Gamma^{\ne}$
is called an {\it asymptotic couple.}\/ Such an asymptotic couple $(\Gamma,\psi)$ is said to be {\it of Hardy type}\/ if 
for all $\alpha,\beta\in\Gamma^{\ne}$:   $\psi(\alpha)>\psi(\beta)$ iff $n\abs{\alpha}<\abs{\beta}$ for all $n$,
 to have {\it small derivation}\/ if for all $\alpha>0$ in $\Gamma$ we have  $\alpha+\psi(\alpha)>0$, and to be {\it short}\/
  if the ordered set $\Gamma$ is short.
The asymptotic couple of any Hardy field is of Hardy type with small derivation (as is that of $\T$).
Rosenlicht \cite[Theorem~3 and Remark~3 following it]{Rosenlicht81} showed conversely that
every asymptotic couple $(\Gamma,\psi)$ of Hardy type with small derivation and $\Gamma$ of finite archimedean rank 
is isomorphic to the asymptotic couple of an analytic Hardy field   containing~$\R$.  Using Corollary~B  we generalize this result\endnote{As a consequence of Corollary~C,
each asymptotic couple of Hardy type with small derivation and   countable rank is isomorphic to the asymptotic couple of a Hardy field extending~$\R$. A different proof of this is given in the forthcoming master's thesis of Clemens Kinn (Univ.~Vienna).}:

\begin{corC}
Let $M$ be a maximal analytic Hardy field. Then any 
short asymptotic couple of Hardy type with small derivation is isomorphic to the asymptotic couple of a spherically complete Hardy subfield of $M$ containing $\R$.
\end{corC}

\noindent
In the rest of this paper we freely use notation and terminology from [ADH] (and its results!). For a summary of relevant material from
[ADH], see   {\it Concepts and Results from \textup{[ADH]}}\/ in the introduction to~\cite{ADHnorm}\endnote{For a list of errata to [ADH],  see \cite{ADHnorm} or    \url{https://tinyurl.com/ADH-errata}.}.
We also refer to various basic facts on shortness from \cite{AD}. As in~\cite{AD} we call an $H$-field {\em closed\/} if it is $\upo$-free, Liouville closed, and newtonian. 

\subsection*{Organization of the paper}
Section~\ref{sec:sof} contains preliminary observations about short ordered fields.
The key point in the proof of Theorem~A is the discussion
 in Section~\ref{sec:altnlc}
to the effect that a certain construction $H\mapsto~H^*$ for real closed $\upo$-free $H$-fields $H$ preserves shortness. This is combined with a more economical way of generating the Newton-Liouville closure of an $\upo$-free $H$-field than in [ADH].
We also need some more routine lemmas to reduce to the case of a real closed $\upo$-free $H$-field. These lemmas are in Section~\ref{sec:proof}, where we complete the proof of Theorem~A and obtain Corollary~B.
In Section~\ref{sec:asc} we prove Corollary~C; this requires Corollary~B and an extension of a
construction from \cite[Section~11]{ADlc}. 

\section{Short Ordered Fields}\label{sec:sof}

\noindent
This section we make two observations on short ordered fields, Lemmas~\ref{shorto} and~\ref{kl}, to be used in the next two sections.
First a reminder about composing valuations.

\subsection*{Composing valuations}  Let
$K$ be a field and $\O_K$ be a valuation ring of $K$, and let~$\pi_K\colon \O_K\to R$ be the residue map onto its residue field $R$. Let also
$\O_R$ be a valuation ring of~$R$. Then $\O:=\pi_K^{-1}(\O_R)$ is a subring of $\O_K$. 

\claim{$\O$  is a valuation ring of $K$.}

\noindent
First, if $a\in K$ and $a\notin \O_K$, then $a^{-1}$ is in the maximal ideal of $\O_K$, which is the kernel of $\pi_K$, and thus $a^{-1}\in \O$. Next, let
$a\in \O_K$ and $a\notin \O$. Then~$a$ does not lie in the maximal ideal of $\O_K$, so $a^{-1}\in \O_K$. Also
$\pi_K(a)\notin \O_R$, so~$\pi_K(a)^{-1}=\pi_K(a^{-1})\in \O_R$, and thus $a^{-1}\in \O$. A useful consequence of the claim is that the maximal ideal of $\O_K$ is a prime ideal of $\O$. 

Let $\pi_R\colon \O_R\to \k$ be the residue morphism onto the residue field $\k$ of $\O_R$, and 
$$\pi\ :=\  \pi_R\circ (\pi_K|_{\O})\ :\  \O \to \O_R\to \k.$$
We identify the surjective ring morphism $\pi\colon \O\to \k$ with the residue map of $\O$ onto its residue field.
(The place $\pi$ is said to be the {\em composition of the places $\pi_K$ and $\pi_R$}.) 
Let now $v\colon K^\times \to \Gamma$ be a valuation on $K$ with valuation ring $\O$, let~$v_K\colon K^\times \to \Gamma_K$ be a valuation on $K$ with valuation ring $\O_K$, and 
$v_R\colon R^\times \to \Gamma_R$ a valuation on $R$ with valuation ring $\O_R$. 
 
It is routine to check that we have an order preserving group embedding $\Gamma_R\to \Gamma$ sending
$v_R(\pi_K(a))$ to $v(a)$ for $a\in \O_K$ with $\pi_K(a)\ne 0$. We identify $\Gamma_R$ with its image in $\Gamma$ via this embedding. The surjective group morphism
$\Gamma\to \Gamma_K$ sending~$v(a)$ to $v_K(a)$ for $a\in K^\times$ is order preserving with kernel $\Gamma_R$. It follows that $\Gamma_R$ is a convex subgroup of $\Gamma$ and 
$\Gamma/\Gamma_R \cong \Gamma_K$ as ordered abelian groups. 

\subsection*{Observations on short ordered fields}
From \cite[Lemma~5.17]{AD} we recall that an ordered abelian group $\Gamma$ is short iff its ordered set $[\Gamma]$ of archimedean classes is short.
From \cite[Corollary~5.18]{AD} we also quote a basic fact about the preservation of shortness under extensions of ordered abelian groups:

\begin{lemma}\label{lem:5.18}
Let $\Delta\subseteq\Gamma$ be an extension of ordered abelian groups. 
\begin{enumerate}
\item[(i)] If $\operatorname{rank}_{\Q}(\Gamma/\Delta)\le \aleph_0$, then  $\Gamma$ is short iff $\Delta$ is short;
\item[(ii)] if $\Delta$ is convex, then $\Gamma$ is short iff $\Delta$ and $\Gamma/\Delta$ are short.
\end{enumerate}
\end{lemma}

\noindent
The following is \cite[Lemma~5.19]{AD}:

\begin{lemma}\label{lem:5.19}
Let $K$ be an ordered field equipped with a convex valuation whose ordered residue field is archimedean. Then $K$ is
short iff its value group is short.
\end{lemma}

\noindent
The next result extends this to nonarchimedean residue fields. 

\begin{lemma}\label{shorto} Let $K$ be an ordered field equipped with a convex valuation.
Then $K$ is short iff
 its ordered residue field $R$ and value group 
$\Gamma_K$ are short. 
\end{lemma}
\begin{proof} 
Suppose first that~$R$ and $\Gamma_K$ are short.
Let $v_K\colon K^\times \to \Gamma_K$ be the given convex valuation on $K$, with valuation ring~$\O_K$
and residue map $\pi\colon \O_K\to R$. Set~$\O_R:=\big\{x\in R: \text{$|x|\le n$  for some $n$}\big\}$, the smallest convex subring of
the ordered field $R$. Then $\O_R$ has archimedean ordered residue field $\k$. Let
$v_R\colon  R^\times\to \Gamma_R$ be a valuation on $R$ with valuation ring~$\O_R$. 

The subring $\O:=\pi^{-1}(\O_R)$ of $\O_K$ is convex. Let $v\colon K^\times \to \Gamma$ be a convex valuation on $K$ with $\O$ as  valuation ring. The considerations in the previous subsection show that we may consider $\k$ as the 
ordered residue field of $\O$, and $\Gamma_R$ as a convex subgroup of $\Gamma$ with $\Gamma/\Gamma_R\cong \Gamma_K$, as ordered abelian groups.

 Now $R$ is short, hence so are $\Gamma_R$ and $\k$ by Lemma~\ref{lem:5.19}. Thus $\Gamma$ is short by Lem\-ma~\ref{lem:5.18}(ii) and so is $K$, using the (convex) valuation $v$ on $K$ and again Lem\-ma~\ref{lem:5.19}. 
 
This shows the ``if'' direction. For the converse, suppose  $K$ is short. 
Then the valuation ring of $K$ and its maximal ideal are  short convex ordered additive subgroups of~$K$, so
$R$ is short by Lemma~\ref{lem:5.18}(ii). The ordered subset~$K^>$ of $K$ is also short, 
  hence so is its image $\Gamma_K$  under the decreasing surjection $v_K$.
\end{proof} 

\begin{lemma}\label{kl} Let $K$ be a short ordered field and $L$ an ordered field extension of countable transcendence degree over $K$. Then $L$ is short. 
\end{lemma}
\begin{proof} We equip $K$ and $L$ with their standard convex valuation whose ordered residue fields are archimedean. Then the value group of $L$ has countable rational rank over the value group of $K$, and the latter being short, so is the former by  Lemma~\ref{lem:5.18}(i).
Hence $L$ is short by Lemma~\ref{lem:5.19}. 
\end{proof}

\section{Alternative Construction of Newton-Liouville Closures}\label{sec:altnlc}

\noindent 
{\it In this section $H$ is an $H$-field with asymptotic couple $(\Gamma, \psi)$.}\/ For $s\in H$ we set $$\Gamma_s\ :=\ \big\{v(s-h^\dagger):\, h\in H^\times\big\}\ \subseteq\ \Gamma_{\infty}.$$
In particular, $\infty\in \Gamma_s$ iff $s\in H^\dagger:=\{h^\dagger:\ h\in H^\times\}$.  

\begin{lemma} \label{ahd} The following are equivalent: \begin{enumerate}
\item[\rm(i)] $H$ is  closed;
\item[\rm(ii)] $H$ is $\upo$-free, real closed, newtonian, and there is no $s\in H$ with $\Gamma_s\subseteq \Psi^{\downarrow}$.
\end{enumerate}
\end{lemma}
\begin{proof} If $H$ is closed, then $H^\dagger=H$, so (ii) holds. Now assume (ii). To derive (i), it suffices to show that $H$ is Liouville closed. Now $H$ is newtonian, so is closed under integration. Let $s\in H$; it is enough to show that then $s\in H^\dagger$. Now $\Gamma_s\not\subseteq \Psi^{\downarrow}$, so we have $h\in H^\times$ with
$s-h^\dagger\in \I(H)$, so $s-h^\dagger\in (1+\smallo)^\dagger$ by [ADH, 14.2.5], and thus $s\in H^\dagger$. 
\end{proof}

\noindent
Next a part of  [ADH, 10.5.20] (replacing $s$, $f$ there by $-s$, $f^{-1}$ if necessary):  

\begin{lemma} \label{hda} Suppose $\Gamma\ne \{0\}$, $H$ is real closed, $s\in H$, and $\Gamma_s\subseteq \Psi^{\downarrow}$. Then there exists an $f$ in an $H$-field extension of $H$ such that:  \begin{enumerate}
\item[\rm(i)] $f$ is transcendental over $H$ and $f^\dagger=s$;
\item[\rm(ii)] the pre-$H$-field extension $H(f)$ of $H$ is an $H$-field with the same constant field as $H$; and
\item[\rm(iii)] for the asymptotic couple $(\Gamma_f, \psi_f)$
of $H(f)$, viewed as an extension of $(\Gamma,\psi)$, we have $vf\in \Gamma_f\setminus \Gamma$, 
$\Gamma_f=\Gamma\oplus \Z vf$, with $\Psi_f$ cofinal in $\Psi$.
\end{enumerate}
\end{lemma}

\noindent
Suppose $H$ is real closed with $\Gamma\ne \{0\}$. Transfinitely iterating the extension procedure of the lemma above, alternating it with taking real closures, and taking unions at limit stages, we obtain a real closed $\d$-algebraic
$H$-field extension $H^*$ of $H$ with asymptotic couple $(\Gamma^*,\psi^*)$ of $H^*$ extending
$(\Gamma, \psi)$, such that: \begin{enumerate}
\item $H^*$ has the same constant field as $H$, and $\Psi$ is cofinal in $\Psi^*$;
\item for all $s\in H$ we have $\Gamma^*_s\not\subseteq \Psi^*{}^{\downarrow}$; and
\item $\Gamma^*=\Gamma \oplus \bigoplus_{i\in I} \Q vf_i$ (internal direct sum) 
where $(f_i)$ is a family of nonzero elements of $H^*$ with $f_i^\dagger\in H$ for all $i\in I$. 
\end{enumerate} 
Given any archimedean class $[\gamma^*]$ with $0\ne\gamma^*\in \Gamma^*$ we choose
 $a\in H^\times$, distinct~$i_1,\dots, i_n\in I$, and $k_1,\dots, k_n\in \Z^{\ne}$, such that
$$af^{k_1}_{i_1}\cdots f^{k_n}_{i_n}\succ 1, \qquad v\big(af^{k_1}_{i_1}\cdots f_{i_n}^{k_n}\big)\in [\gamma^*].$$ 
Associating to $[\gamma^*]$ the element 
$(af^{k_1}_{i_1}\cdots f^{k_n}_{i_n})^\dagger=a^\dagger+k_1f^\dagger_{i_1}+\cdots + k_nf^\dagger_{i_n}$ of $H^{>}$, we obtain a strictly increasing map $[(\Gamma^*)^{\ne} ]\to H^{>}$, by [ADH, 10.5.2].
Recall that the ordered residue field of $H$ is isomorphic to the ordered constant field of $H$.
It follows that if $H$ is short, then so are the ordered residue field of $H$ 
and the ordered set $[\Gamma^*]$,  hence $\Gamma^*$ is short as well by the remark before Lemma~\ref{lem:5.18}. 
Thus if $H$ is short, then so is $H^*$  by Lemma~\ref{shorto}. 

The process leading from $H$ to $H^*$ can now be applied to $H^*$ instead of $H$, and iterating this process
$\omega$ times, and taking a union gives us a real closed
$H$-field extension $H^{\#}$ of $H$ with asymptotic couple $\big(\Gamma^{\#}, \psi^{\#}\big)$ extending 
$(\Gamma, \psi)$ such that: \begin{enumerate}
\item[(4)] $H^{\#}$ has the same constant field as $H$, and $\Psi$ is cofinal in $\Psi^{\#}$;
 \item[(5)] there is no $s\in H^{\#}$ with $\Gamma^{\#}_s\subseteq \Psi^{\#\downarrow}$; and
 \item[(6)]  if $H$ is short, then so is $H^{\#}$.
 \end{enumerate} 
Let now $H$ be $\upo$-free and real closed. We build a sequence $(H_n)$ of $\upo$-free
real closed $\d$-algebraic $H$-field extensions of $H$, with
$H_{n+1}$ extending $H_n$ for all $n$: \begin{itemize}
\item $H_0:=H$,
\item for even $n$, $H_{n+1}:= H_n^{\#}$, 
\item for odd $n$, $H_{n+1}$ is a maximal immediate $\d$-algebraic $H$-field extension of~$H_n$ (so $H_{n+1}$ is newtonian by [ADH, 10.5.8, 14.0.1]).
\end{itemize} 
Then $H_{\infty}:= \bigcup_n H_n$ is closed by Lemma~\ref{ahd}, in view of (5) above. Since $H_{\infty}$ is also $\d$-algebraic over $H$ and has the same constant field as $H$, it follows that $H_{\infty}$ is a Newton-Liouville closure of $H$, by  [ADH, pp.~669, 685]. 

Now shortness is inherited by
immediate extensions of $H$-fields, so if $H$ is short, then so is $H_{\infty}$ in view of (6) above. 
 This leads to:

\begin{cor}\label{corupo} Suppose $H$ is $\upo$-free, real closed, and short,  and $E$ is a differentially algebraic $H$-field extension of $H$ with the same constant field as $H$. Then $E$ is short.
\end{cor}
\begin{proof} Take a Newton-Liouville closure of $E$. This is also a Newton-Liouville closure of $H$ and
  isomorphic to the $H_{\infty}$ above, by  [ADH, 16.2.1]. So this Newton-Liouville closure
of $E$ is short, and thus $E$ is short.  
\end{proof}

\section{Proof of the Main Theorem}\label{sec:proof}

\noindent
We still need a few generalities about preserving shortness:

\subsection*{From pre-$H$-fields to $\upo$-free $H$-fields} By [ADH, p.~445], a pre-$H$-field $K$ has a ``smallest'' $H$-field extension,  the {\em $H$-field hull $\operatorname{H}(K)$}\/ of $K$, whose underlying valued differential field is the pre-$\d$-valued hull $\dv(K)$ of $K$. 

\begin{lemma}\label{hfh} Let $K$ be a short pre-$H$-field. Then $\operatorname{H}(K)$ is short.
\end{lemma}
\begin{proof} By  [ADH, 10.3.2] and Lemma~\ref{lem:5.18}(i), the value group of $\operatorname{H}(K)$ is short, and by [ADH, remarks preceding   10.3.2]  its ordered residue field equals that of $K$, and so is short as well. Hence $\operatorname{H}(K)$ is short by Lemma~\ref{shorto}.
\end{proof} 

\noindent
The next lemma has almost the same proof as \cite[Corollary 2.18]{AD}. It is used for Hardy fields that do not contain 
$\R$ and might therefore not be $H$-fields. 

\begin{lemma}\label{hfs} If the Hardy field $H$ is short, then so is the Hardy field $H(\R)$.
\end{lemma}
\begin{proof} This is clear if $H\subseteq \R$. Assume $H$ is a short Hardy field and $H\not\subseteq \R$. Let~$E$ be the $H$-field hull of~$H$, taken as an $H$-subfield of the Hardy field extension~$H(\R)$ of~$H$. Then $E$ is short by 
Lemma~\ref{hfh}.  
Now use that $H(\R)= E(\R)$ and $\Gamma_{E(\R)}=\Gamma_E\neq\{0\}$ by 
 [ADH, 10.5.15 and remark preceding 4.6.16].
\end{proof}

\noindent
Elaborating the proof of  [ADH, 11.5.15] we now show:

\begin{lemma}\label{ksl} Let $K$ be a short $H$-field and $L$ a Schwarz closed $H$-field extension of $K$. Then there exists a
 short $\upo$-free $H$-subfield of $L$ containing $K$.
\end{lemma}
\begin{proof} If $K$ is grounded, this follows from  [ADH, 11.7.17] and Lemma~\ref{kl}.  Suppose~$K$ is 
ungrounded. Then $K$ has a gap or has asymptotic integration by [ADH, 9.2.16].  If $K$ has a gap $vs$ with $s\in K$, then we take
$y\in L$ with $y\nasymp 1$ and $y'=s$, so that by [ADH, 10.2.1, 10.2.2 and subsequent remarks], $K(y)$ is a grounded $H$-subfield of $L$, and we are back to the grounded case. In general we arrange, by passing to the real closure of $K$ in $L$, that the value group $\Gamma_K$ of the valued field~$K$ is divisible. For such $K$ it remains to consider the case that $K$ has asymptotic integration (and thus rational asymptotic integration) and is not $\upo$-free.  Assume we are in this case. We distinguish two subcases:

\subcase[1]{$K$ is not $\upl$-free.}

\medskip
\noindent Then  [ADH, 11.5.14] yields $s\in K$ creating a gap over $K$, with $S:=\big\{v(s-a^\dagger):\, a\in K^\times\big\}$ a cofinal subset of $\Psi_K^{\downarrow}$ and $s\ne 0$ by [ADH, 11.5.13]. Taking $f\in L^\times$ with~$f^\dagger=s$, $K(f)$ has a gap by [ADH, remark preceding   11.5.15].
Moreover,  $K(f)$ is an $H$-subfield of $L$ by  [ADH, 10.4.5(iv)], so $K(f)$ falls under the ``gap'' case treated earlier. 

\subcase[2]{$K$ is $\upl$-free.} 

\medskip
\noindent As in  [ADH,   11.5, 11.7] we have 
the pc-sequence $(\upo_{\rho})$ in $K$. Now $K$ is not $\upo$-free, so we can take $\upo\in K$ such that 
$\upo_{\rho}\leadsto \upo$.  The Schwarz closed $H$-field $L$ has a unique expansion to a $\HLO$-field $\boldsymbol L$, and we let $\boldsymbol K=(K, I, \Lambda, \Omega)$ be the $\HLO$-field expansion of $K$ such that $\boldsymbol K \subseteq \boldsymbol L$.  We are now in the situation of [ADH, 16.4.6], whose proof
gives two possibilities: 

(a): $\Omega=\omega(K)^{\downarrow}$. As  in the proof of that lemma for {\sc Case 1} this yields a pre-$H$-field extension $K\<\upgamma\>$ of $K$ with a gap that embeds over $K$ into $L$. The residue field of $K\<\upgamma\>$ equals that of $K$, and so $K\<\upgamma\>$ is an $H$-field by [ADH, 9.1.2]. Moreover, $K\<\upgamma\>$ is short by Lemma~\ref{kl}, so $K\<\upgamma\>$ falls under the ``gap''  case  treated earlier. 

(b): $\Omega=K\setminus \sigma\big(\Upg(K)\big){}^{\uparrow}$. As in {\sc Case 2} in the proof of  [ADH,  16.4.6]
this yields an immediate pre-$H$-field extension $K(\upl)$ of $K$ that is not $\upl$-free and that embeds over $K$ into $L$. Then $K(\upl)$ is an $H$-field by  [ADH, 9.1.2] and is short, so $K(\uplambda)$ falls 
under Subcase 1. 
\end{proof}

\subsection*{Finishing the proof} The following is a bit more general than Theorem~A:

\begin{theorem}\label{main, generalized} Let $H$ be a short pre-$H$-field and $E$ a  pre-$H$-field extension of $H$ and $\d$-algebraic over $H$ with short ordered residue field. Then $E$ is short.
\end{theorem}
\begin{proof}  Expand $E$ to a pre-$\HLO$-field 
$\boldsymbol E$. The proof of  [ADH,  16.4.9] yields a Newton-Liouville closure $\boldsymbol L=(L,\dots)$ of $\boldsymbol E$ such that $L$ is $\d$-algebraic over $E$ and thus over~$H$, and
the ordered residue field of $L$ is a real closure of the ordered residue field of $E$. In particular, the ordered constant field $D$ of $L$, being isomorphic to the ordered residue field of $L$, is short.  It is enough to prove that $L$ is short. Let $K$ be the $H$-field hull of $H$ in $L$. Then $K$ is short by Lemma~\ref{hfh}, and
so Lemma~\ref{ksl} yields a short $\upo$-free $H$-subfield $F\supseteq K$ of $L$. Now $F(D)$ is an $H$-subfield of
$L$ and has the same (short) value group as $F$, by  [ADH, 10.5.15], and short constant field~$D$. So $F(D)$ is short by Lemma~\ref{shorto}.
Hence the real closure of $F(D)$ in $L$ is short and $\upo$-free. Thus $L$ is short by Corollary~\ref{corupo}.
\end{proof}

\subsection*{Embeddings into closed $\eta_1$-ordered $H$-fields} Maximal Hardy fields, maximal smooth Hardy fields, and maximal analytic Hardy fields are closed $H$-fields, and are $\eta_1$-ordered by \cite[Theorem~A]{ADHfgh} and \cite[Theorem~A and subsequent remark]{AD}. More generally, 
{\it in this subsection $L$ is a closed $\eta_1$-ordered $H$-field.}\/ Using
Theorem~A from the introduction we generalize \cite[Proposition~7.6]{AD}:

\begin{prop}\label{prop:7.6}
Let $E$ be an $\upo$-free pre-$H$-field and $K$ be a short pre-$H$-field extending~$E$ such that $\res(E)=\res(K)$. Then any embedding
$E\to L$  extends to an embedding  $K\to L$.
\end{prop}
\begin{proof}
By Lemma~\ref{hfh} and \cite[remark after Lemma~5.19]{AD}, the  real closure $\operatorname{H}(K)^{\operatorname{rc}}$ of the $H$-field hull of $K$ is short.
Moreover, the $H$-field $\operatorname{H}(E)^{\operatorname{rc}}$ is $\upo$-free by [ADH, 13.6.1], and each embedding $E\to L$ extends to an embedding $\operatorname{H}(E)^{\operatorname{rc}}\to L$.
The ordered residue field of $\operatorname{H}(E)^{\operatorname{rc}}$ and of~$\operatorname{H}(K)^{\operatorname{rc}}$ is the real closure of
$\res(E)=\res(K)$, cf.~[ADH, remark before 10.3.2].
Hence replacing $K$ by $\operatorname{H}(K)^{\operatorname{rc}}$ and then~$E$ by~$\operatorname{H}(E)^{\operatorname{rc}}$, taken  inside $\operatorname{H}(K)^{\operatorname{rc}}$,
we arrange~$E$,~$K$ to be $H$-fields with real closed constant fields~$C_E=C_K$.

Next expand $K$ to a $\HLO$-field $\mathbf K$, and let $\mathbf M=(M,\dots)$ be 
 a Newton-Liouville closure   of $\boldsymbol K$ such that $M$ is $\d$-algebraic over $K$  and $C_M=C_K$. (See the proof of Theorem~\ref{main, generalized}.)  Then $M$ is closed, and  short by Theorem~A, so by  \cite[Proposition~7.6]{AD} any embedding $E\to L$ extends to an embedding $M\to L$. 
\end{proof}

\noindent
Next a generalization of \cite[Lemma~7.8]{AD}, where we recall that a valued differential field is said to have {\em very small derivation\/} if $\mathcal{O}'\subseteq \smallo$ (with $\mathcal{O}$ and $\smallo$ as usual). 

\begin{lemma}\label{lem:7.8} Suppose $L$ has small derivation 
and $C_L=\R$. Then any short pre-$H$-field with very small derivation and ar\-chi\-me\-dean residue field embeds into $L$.
\end{lemma}
\begin{proof} Let $E$ be a short pre-$H$-field with very small derivation and ar\-chi\-me\-dean residue field. To embed $E$ into $L$ we pass to $\operatorname{H}(E)^{\operatorname{rc}}$ to arrange that $E$ is a real closed $H$-field with small derivation. Now Theorem~A  yields a closed short $H$-field extension $K$ of $E$ with $C_E=C_K$.
Then \cite[Lemma~7.8]{AD} yields an embedding of $K$ (and thus of $E$) into $L$.
\end{proof}

\noindent
By \cite[Proposition~13.11]{ADH4} the universal part of the theory of Hardy fields, viewed as structures
in the language (specified there) of ordered valued differential fields, is the theory of pre-$H$-fields with very small derivation.
The following complements this result and includes Corollary~B from the introduction:

\begin{cor}\label{cor:7.8}
Any short pre-$H$-field with very small derivation and archimedean residue field embeds into every maximal Hardy field. Likewise with ``maximal'' replaced by  ``maximal smooth'', respectively
``maximal analytic''. 
\end{cor}

\noindent
This is a consequence of Lemma~\ref{lem:7.8}. We finish this section with an application to the $H$-field 
$\No(\omega_1)$ of surreal numbers of countable length equipped with the Berarducci-Mantova derivation~\cite{BM}. 
(Note, however,  that the Continuum Hypothesis gives {\em isomorphism} results~\cite{AD, ADHfgh} stronger than
Corollary~\ref{38}.)

\begin{cor}\label{38}
Let $M$ be a maximal Hardy field. Then $\No(\omega_1)$ embeds  into $M$.
Likewise with   ``maximal smooth'' and ``maximal analytic'' in place of ``maximal''. 
\end{cor}
\begin{proof} The $H$-field $\No(\omega_1)$ is exhibited in \cite[Remark after Corollary 4.5]{ADH2} as the increasing union of grounded  and short $H$-subfields
$K_{\varepsilon}$, with $\varepsilon$ ranging over the countable $\varepsilon$-numbers.
Let $\alpha$, $\beta$ range over countable infinite limit ordinals, and set 
$$E_{\alpha}\ :=\ \bigcup_{\varepsilon<\varepsilon_{\alpha}} K_{\varepsilon}.$$ 
Then $E_{\alpha}$ is a short $H$-subfield of $\No(\omega_1)$ with constant field $\R$, and is $\upo$-free by~[ADH, 11.7.15].  Thus $E_{\omega}$ embeds into $M$ by Corollary~\ref{cor:7.8}. Transfinite recursion on $\alpha$
using Proposition~\ref{prop:7.6} then yields for all $\alpha$
an embedding  $h_\alpha\colon  E_\alpha\to M$ such that~$h_\alpha=h_\beta|_{E_\alpha}$ for $\alpha\le\beta$.
It follows that the common extension of the $h_\alpha$ embeds~$\No(\omega_1)$ into $M$ as required.
\end{proof}

\section{Constructing Hardy Fields with Given Asymptotic Couple}\label{sec:asc}

\noindent
In this section $\k$ is an ordered field, and we consider $H$-couples $(\Gamma,\psi)$ over $\k$ as defined in 
\cite{ADH3}. 
We revisit a construction 
from \cite[Section~11]{ADlc}\endnote{In \cite{ADcl,ADlc} an $H$-couple~$(\Gamma,\psi)$ satisfies additional requirements:  $\psi(\gamma)=\gamma$ for some $\gamma>0$ in~$\Gamma$, and~$(\Gamma, \psi)$ is of Hahn type.} which associates to each $H$-couple~$(\Gamma,\psi)$ of Hahn type over~$\k$
 satisfying (+) a spherically complete $H$-field  with constant field $\k$ 
and $H$-couple $(\Gamma,\psi)$ over $\k$, and closed under powers. Here (+) is the condition that~${\psi(\gamma)=\gamma}$ for some $\gamma>0$ in $\Gamma$, and that $\Gamma$ admits a valuation basis for its $\k$-valuation. 
In the first subsection we carry out this construction without assuming (+):
Corollary~\ref{cor:realizing H-cples}. 
This leads to Corollary~C from the introduction. 
The construction involves equipping suitable Hahn fields with the ``right''  derivation. Other explorations of derivations on Hahn fields are in \cite{KM,Schm}.

\subsection*{Generalizing \cite[Section~11]{ADlc}} 
Let $\Gamma$ be an ordered vector space over $\k$ and suppose it is a spherically complete Hahn space over~$\k$; see [ADH,  2.4] for the definitions. 
Using [ADH, 2.3.2, 2.4.23] we identify~$\Gamma$ with the Hahn product $H[I,\k]$
where~$I:=\big(\text{$[\Gamma^{\ne}]_{\k}$  with  reversed ordering}\big)$, and we let~$i$,~$j$  range over $I$.
The elements of $\Gamma$ are thus the $\gamma=(\gamma_i)\in \k^I$ whose support~$\supp \gamma =  \{ i: \gamma_i\ne 0 \}$
is a well-ordered subset of~$I$. 
Let  $e_i=(e_{ij})\in\Gamma$ be given by~$e_{ij}=0$ if $i\ne j$ and~$e_{ii}=1$.
Then~$e_i>0$ for each $i$, and
the map~$i\mapsto [e_i]_{\k}\colon I\to [\Gamma^{\ne}]_{\k}$ is decreasing and bijective (so $\{e_i:i\in I\}$ is valuation-independent with
respect to the $\k$-valuation on the ordered $\k$-vector space $\Gamma$).
We think of each~$\gamma=(\gamma_i)\in \Gamma$ as an infinite sum
$$\gamma\  =\  \sum_i \gamma_i e_i.$$ 
Let $\alpha$, $\beta$, $\gamma$ range over $\Gamma$.
We say that {\bf $e_i$ occurs in~$\gamma$} if $i\in\supp\gamma$. 
Thus the set of ``Hahn basis elements'' $e_i$ that occur in $\gamma$ is reverse well-ordered,
and for~$\gamma\ne 0$ and $i_0=\min\supp\gamma$ (so
$e_{i_0}$ is  the largest Hahn basis element occurring in $\gamma$),
we have~$[\gamma]_{\k} = [e_{i_0}]_{\k}$, and the equivalence  $\gamma>0\Leftrightarrow \gamma_{i_0}>0$. Note:
$i=[e_i]_{\k}$. 

Let also $\psi\colon \Gamma^{\ne}\to\Gamma$ be such that $(\Gamma,\psi)$  is an $H$-couple over~$\k$,
and assume it is of Hahn type,
that is, for all~$\alpha,\beta\ne 0$ with $\psi(\alpha)=\psi(\beta)$ there is a $c\in\k^\times$
with~$\psi(\alpha-c\beta)>\psi(\alpha)$;
see \cite{ADH3} or~\cite[Section~8]{ADHfgh}, also for the fact that then
$$\psi(\alpha)\le \psi(\beta)\ \Longleftrightarrow\  [\alpha]_{\k}\ge[\beta]_{\k}\qquad(\alpha,\beta\ne 0),$$
and thus $i<j\Leftrightarrow \psi(e_i)<\psi(e_j)$.  Let $t^\Gamma$ be a multiplicative copy of the (additive) ordered abelian group $\Gamma$, ordered such that
$\gamma\mapsto t^\gamma\colon\Gamma\to t^\Gamma$ is an order-reversing isomorphism, and consider the valued ordered Hahn field
$K:=\k(\!(t^\Gamma)\!)$ over $\k$, cf.~[ADH, 3.1, 3.5]. Its elements are the formal series
$$f=\sum_\gamma f_\gamma t^\gamma \qquad (f_\gamma\in\k)$$
whose support
$\supp f =  \{ \gamma: f_\gamma\ne 0  \}$
is a well-ordered subset of $\Gamma$. 
Recall   that a family $(f_\lambda)_{\lambda\in\Lambda}$ of elements of $K$ is said to be {\it summable}\/ if
$\bigcup_\lambda \supp f_\lambda$ is well-ordered and for each $\gamma$ there are only finitely many
$\lambda\in\Lambda$ such that~$f_{\lambda,\gamma}\ne 0$; in this case we define~$\sum_\lambda f_\lambda$ to be the
series $f\in K$ with $f_\gamma=\sum_\lambda f_{\lambda,\gamma}$ for each~$\gamma$~[ADH, p.~712].
Also recall from [ADH, p.~713] that a map $\Phi\colon K\to K$ is said to be {\it strongly additive}\/
if for every summable family $(f_\lambda)$ in $K$ the family
$\big(\Phi(f_\lambda)\big)$ is summable and $$\Phi\left(\sum_\lambda f_\lambda\right)\ =\ \sum_\lambda\Phi(f_\lambda).$$

\begin{lemma}\label{lem:summable}
Let $S\subseteq\Gamma$ be well-ordered. Then
\begin{enumerate}
\item[(i)] for each $\gamma$ there are only finitely many $\alpha\in S$ such that $\gamma=\alpha+\psi(e_i)$ for some $e_i$ occurring in $\alpha$;
\item[(ii)] the set of all $\alpha+\psi(e_i)$ with $\alpha\in S$ and $e_i$ occurring in $\alpha$ is well-ordered.
\end{enumerate}
\end{lemma}
\begin{proof}
For (i), suppose $\gamma=\alpha+\psi(e_i)=\beta+\psi(e_j)$ for elements $\alpha<\beta$ in $S$, with~$e_i$,~$e_j$ occurring in $\alpha$, $\beta$, respectively.
Then $\psi(e_i)-\psi(e_j) = \beta-\alpha > 0$, so~$[e_i]_{\k} < [e_j]_{\k}$ and $[\beta-\alpha]_{\k} = \big[ \psi(e_i)-\psi(e_j) \big]_{\k} < [e_i-e_j]_{\k} = [e_j]_{\k}$. Hence $e_j$ occurs in $\alpha$. Thus if we have a strictly increasing sequence $(\alpha_n)$ in $S$ and
a sequence~$(i_n)$ in $I$ such that  $e_{i_n}$ occurs in $\alpha_n$ and $\alpha_n + \psi(e_{i_n}) = \alpha_{n+1} + \psi(e_{i_{n+1}})$, for all $n$, then all $e_{i_n}$ occur in $\alpha_0$ and $(e_{i_n})$ is strictly increasing, contradicting that the set of $e_i$ occurring in $\alpha_0$ is reverse well-ordered.

For (ii), suppose towards a contradiction   that $(i_n)$ is a sequence in $I$ and $(\alpha_n)$ is a sequence in $S$ such that
$e_{i_n}$ occurs in $\alpha_n$ and $\alpha_n+\psi(e_{i_n}) > \alpha_{n+1}+\psi(e_{i_{n+1}})$ for all $n$.
Passing to a subsequence and using the well-orderedness of $S$  we arrange that $\alpha_n \le \alpha_{n+1}$ for all $n$.
Then
$0 \le \alpha_{n+1}-\alpha_0 < \psi(e_{i_0})-\psi(e_{i_{n+1}})$, so
$$[\alpha_{n+1}-\alpha_0]_{\k}\ \le\ \big[ \psi(e_{i_0})-\psi(e_{i_{n+1}}) \big]_{\k}\ <\  [e_{i_0}-e_{i_{n+1}}]_{\k}\ =\  [e_{i_{n+1}}]_{\k}.$$
 Thus all $e_{i_n}$ occur in $\alpha_0$.
Also $\psi(e_{i_{n+1}}) < \psi(e_{i_n}) + (\alpha_n-\alpha_{n+1}) \le \psi(e_{i_n})$ and so~$e_{i_{n+1}} > e_{i_n}$ for all $n$,  contradicting that the set of $e_i$ occurring in $\alpha_0$ is  reverse well-ordered.
\end{proof}

\noindent
For $\alpha=\sum_i \alpha_i e_i$ ($\alpha_i\in\k$) the set of $e_i$ with $\alpha_i\ne 0$ is reverse well-ordered, so
the subset  $\big\{ \alpha+\psi(e_i): \text{$e_i$ occurs in $\alpha$}\big\}$ of $\Gamma$ is well-ordered and is the support of~$(t^\alpha)' := -\sum_i \alpha_i\, t^{\alpha+\psi(e_i)}\in K$.
Let $f=\sum_{\alpha}f_{\alpha}t^\alpha$ range over $K$. Then the family~$\big(f_\alpha (t^\alpha)'\big)$ is summable by Lemma~\ref{lem:summable}, and we put
$$f'\  :=\  \sum_\alpha f_\alpha (t^\alpha)'\in K \quad(\text{so $f'=0$ for $f\in \k$}). $$
Using $(t^{\alpha+\beta})'=-\sum_i (\alpha_i+\beta_i) t^{\alpha+\beta+\psi(e_i)} = (t^\alpha)' t^\beta+ t^\alpha (t^\beta)'$ and  \cite[Corollary~3.9]{vdH:noeth}
 one verifies that the map $f\mapsto f'\colon K\to K$ is a strongly additive $\k$-linear derivation on~$K$. 
 For $\alpha\ne 0$ we set $\alpha':=\alpha+\psi(\alpha)$. 

\begin{lemma}\label{lem:leading term}
If $0\ne f \prec 1$, then 
$$f'\sim - \alpha_{i} f_\alpha t^{\alpha'}\qquad\text{where $\alpha:=vf$, $i:=[\alpha]_{\k}$.}$$
Hence if $0\ne f\nasymp 1$, then  $v(f') = (vf)'$.
\end{lemma}
\begin{proof}
 For $\alpha\ne 0$ and $i:=[\alpha]_{\k}$ we have $i=[e_{i}]_{\k}$,
 hence $(t^\alpha)' \sim -\alpha_{i}t^{\alpha+\psi(e_{i})} = -\alpha_{i} t^{\alpha'}$.
If 
 $0<\alpha<\beta$, then $\alpha'<\beta'$ and thus
 $(t^\alpha)' \succ (t^\beta)'$. This yields the first implication.
 For the second, let $0\ne f\nasymp 1$, arrange~$f \prec 1$ by replacing $f$ with $1/f$ if necessary, and use the first part of the lemma.
 \end{proof}

\begin{lemma}\label{lem:der}
The derivation $f\mapsto f'$   makes $K$ into an $H$-field with constant field~$\k$ and
asymptotic couple $(\Gamma,\psi)$.
\end{lemma}
\begin{proof}
Let $f\in K\setminus\k$. Then $f'\ne 0$:  arranging $0\ne f\nasymp1$ by subtracting~${f_0\in \k}$ from~$f$ in case
$f\asymp 1$, we have $v(f')=(vf)'$ by Lemma~\ref{lem:leading term}, in particular, ${f'\ne 0}$. Thus the constant field   of our derivation is~$\k$.
The valuation ring $\mathcal O$ of~$K$ is the convex
hull of~$\k$ in $K$, and $\mathcal O=\k+\smallo$.
Finally, suppose $0<f\prec 1$. Then~${\alpha:=vf>0}$, so $\alpha_{i}>0$ for $i:=[\alpha]_{\k}$. Also $f_\alpha>0$
and $f'\sim - \alpha_{i} f_\alpha t^{\alpha'}$ and thus $f'<0$.
\end{proof}

\noindent
An $H$-field $H$ with constant field $C$ is said to be {\it closed under powers}\/ if $H^\dagger$
is a $C$-linear subspace of $H$. In that case,
 the asymptotic couple of  $H$   is an $H$-couple over $C$ of Hahn type
in a natural way, by \cite[Lemma~7.4, Proposition~7.5]{ADlc}. See  \cite[Sections~7,~8]{ADlc} for  other basic facts about $H$-fields
closed under powers. 

\begin{prop}\label{prop:ADlc}
The $H$-field $K$ is closed under powers, and its associated $H$-couple over $\k$ is $(\Gamma,\psi)$.
\end{prop}

\noindent
This follows from Lemma~\ref{lem:der} using the strong additivity of the derivation $f\mapsto f'$ of $K$
and the fact that $(t^{c\alpha})^\dagger=c(t^\alpha)^\dagger$ for all $\alpha$ and all $c\in\k$, just like in \cite{ADlc}, Proposition~11.4 followed
from Lemmas~11.1, 11.2.

\begin{remarkNumbered}\label{rem:restrict}
Let $\Gamma_0$ be a subgroup of $\Gamma$ with $\Psi=\psi(\Gamma^{\ne})\subseteq\Gamma_0$.
Then~$K_0:=\k(\!(t^{\Gamma_0})\!)$ is an $H$-subfield of $K$ with constant field $\k$ and
asymptotic couple $(\Gamma_0,\psi|_{\Gamma_0^{\ne}})$.
If~$\Gamma_0$ is a $\k$-linear subspace of $\Gamma$, then $K_0$ is closed under powers, and
its associated $H$-couple over $\k$ is $(\Gamma_0,\psi|_{\Gamma_0^{\ne}})$.
\end{remarkNumbered}  

\noindent
The construction above generalizes that of the derivation defined in \cite{DHK} on the 
$H$-field of logarithmic hyperseries:

\begin{example}
With $\k=\R$, let $I$ be an ordinal and  $\Gamma:=H[I,\R]=\R^I$. Let $i$ range over~$I$ and take~${\psi\colon\Gamma^{\ne}\to\Gamma}$ to be constant on archimedean classes with~$\psi(e_i)=\sum_{j\le i} e_j$ for all $i$. Then~$(\Gamma,\psi)$ is an $H$-couple over $\R$ of Hahn type.
Note that $\psi(e_0)=e_0$ is the smallest element of $\Psi=\psi(\Gamma^{\ne})$, and~$(\Gamma,\psi)$  has gap~$\sum_i e_i$.
The construction above yields the $H$-field $K:=\R(\!(t^\Gamma)\!)$ with constant field $\R$. This $K$ is closed under powers and has associated $H$-couple $(\Gamma,\psi)$ over $\R$.  Changing notation, let $I$ be the ordinal $\alpha$ and
$\mathbb L_{<\alpha}=\R[[\mathfrak L_{<\alpha}]]$ the $H$-field of  logarithmic hyperseries defined in
\cite[Sections 1, 3]{DHK}. Then the unique strongly additive field isomorphism $h\colon K\to \mathbb L_{<\alpha}$ over~$\R$ with~$h(t^\gamma)=\prod_i \ell_i^{-\gamma_i}$ for $\gamma=\sum_i \gamma_i e_i$ is an isomorphism of $H$-fields.  
\end{example}

\noindent
So far the ordered vector space $\Gamma$ over $\k$ was a spherically complete Hahn space over $\k$ and $(\Gamma, \psi)$ an $H$-couple over $\k$ of Hahn type. We now relax this:  {\em for  the rest of this subsection $(\Gamma,\psi)$ is an arbitrary
$H$-couple over $\k$.}\/
We shall need a variant of \cite[Lemma~3.2]{ADcl}:

\begin{lemma}\label{lem:3.2} 
Let   $\Gamma_1$ be an ordered vector space over $\k$
extending $\Gamma$ with~$[\Gamma]_{\k}=[\Gamma_1]_{\k}$. Then there is a unique map~$\psi_1\colon\Gamma_1^{\ne}\to\Gamma_1$
such that $(\Gamma_1,\psi_1)$ is an $H$-couple over $\k$ extending $(\Gamma,\psi)$.
If $(\Gamma,\psi)$ is of Hahn type and $\Gamma_1$ is a Hahn space over $\k$, then $(\Gamma_1,\psi_1)$ is of Hahn type.
\end{lemma}
\begin{proof}
To verify axiom (A3) of asymptotic couples, use that for distinct $\alpha,\beta\ne0$ we have $\big[\psi(\alpha)-\psi(\beta)\big]_{\k}<[\alpha-\beta]_{\k}$, by \cite[p.~536]{ADH3}.
\end{proof}

\noindent
Now assume $(\Gamma,\psi)$ is of Hahn type.
By the Hahn Embedding Theorem for Hahn spaces [ADH, 2.4.23] we may view $\Gamma$ as an ordered $\k$-linear subspace of the ordered vector space
$\hat\Gamma:=H[I,\k]$ over $\k$, where~$I:=\big([\Gamma^{\ne}] \text{ with reversed ordering}\big)$.    Lemma~\ref{lem:3.2} yields a unique
map~${\hat\psi\colon \hat\Gamma^{\ne}\to\hat\Gamma}$ making $(\hat\Gamma,\hat\psi)$ an $H$-couple over $\k$   
  extending $(\Gamma,\psi)$. Then  $(\hat\Gamma,\hat\psi)$ is   of Hahn type.
Let~$\hat K:=\k(\!(t^{\hat \Gamma})\!)$ be the $H$-field with constant field $\k$ and $H$-couple~$(\hat\Gamma,\hat\psi)$
over~$\k$ and  closed under powers that was constructed above, with $\hat K$,   $(\hat\Gamma,\hat\psi)$ in the roles of~$K$,~$(\Gamma,\psi)$, respectively. 
Then $\hat K$ has the $H$-subfield~$K:=\k(\!(t^\Gamma)\!)$ which is closed under powers and has 
$(\Gamma,\psi)$ as its $H$-couple  over $\k$,
by Remark~\ref{rem:restrict}.
This shows:

\begin{cor}\label{cor:realizing H-cples}
Every $H$-couple over $\k$ of Hahn type  is the $H$-couple of a spherically complete $H$-field with constant field $\k$ and
closed under powers.
\end{cor}

\subsection*{Proof of Corollary~C}
 {\it In this subsection $L$ is a closed $\eta_1$-ordered $H$-field with small derivation and constant field $\R$.}\/ 
Every ordered vector space over $\R$ is a Hahn space over $\R$,
so an $H$-couple $(\Gamma,\psi)$ over $\R$ is of Hahn type iff 
it is of Hardy type.
Thus by Lem\-ma~\ref{lem:7.8} and Corollary~\ref{cor:realizing H-cples}:

\begin{lemma}\label{lem:C1}
If $(\Gamma,\psi)$ is a short $H$-couple over $\R$ of Hardy type with small derivation, 
then $(\Gamma,\psi)$ is isomorphic to the $H$-couple over~$\R$ of a spherically complete $H$-subfield of $L$ containing $\R$ and closed under powers.
\end{lemma}

\noindent
In the same way that Lemma~\ref{lem:7.8} gave rise to Corollary~\ref{cor:7.8},   Lemma~\ref{lem:C1} yields:

\begin{cor}
If $M$ is a maximal Hardy field,
then every short $H$-couple over~$\R$ of Hardy type with small derivation is isomorphic to the
$H$-couple over $\R$ of a spherically complete $H$-subfield of $M$ containing $\R$ and closed under powers.
Likewise with ``maximal smooth'' $($respectively,``maximal analytic''$)$ in place of ``maximal''.
\end{cor}



\noindent
The following is clear from [ADH, 9.8.1].

\begin{lemma}\label{lem:9.8.1}
Let $(\Gamma,\psi)$ be an $H$-asymptotic couple and
  $\Gamma^*$ be an ordered vector space over $\R$ containing $\Gamma$ as an ordered subgroup with $[\Gamma]=[\Gamma^*]$.
Then there is a unique map $\psi^*\colon (\Gamma^*)^{\neq}\to\Gamma^*$ extending $\psi$ which makes $(\Gamma^*,\psi^*)$ an $H$-asymptotic couple. Moreover, $(\Gamma^*,\psi^*)$ is an $H$-couple over $\R$, and $(\Gamma^*,\psi^*)$ is of Hahn type iff~$(\Gamma,\psi)$ is of Hardy type. 
\end{lemma}

\begin{lemma}\label{shortsph}
Let    $(\Gamma,\psi)$ be a short asymptotic couple of Hardy type with small derivation.
  Then~$(\Gamma,\psi)$   is isomorphic to the asymp\-to\-tic couple of a spherically complete $H$-subfield of~$L$ containing~$\R$.
\end{lemma}
\begin{proof} 
The  Hahn product $\Gamma^*:=H[I,\R]$, with
$I\ :=\ \big(\text{$[\Gamma^{\ne}]$  with reversed ordering}\big)$
is an ordered vector space over $\R$; see
 [ADH, p.~98].
The Hahn Embedding Theorem~[ADH, 2.3.4, 2.4.18, 2.4.19] yields an ordered group embedding $\iota\colon \Gamma\to\Gamma^*$ 
such that  $\big[\iota(\Gamma)\big]=[\Gamma^*]$. Identify $\Gamma$ with its image in $\Gamma^*$ via $\iota$.  Lemma~\ref{lem:9.8.1} gives a unique extension
$\psi^*\colon (\Gamma^*)^{\neq}\to\Gamma^*$ of~$\psi$ such that~$(\Gamma^*,\psi^*)$ is an $H$-asymptotic couple.  Then $(\Gamma^*,\psi^*)$  is an $H$-couple~over $\R$
  of Hahn type,  
  and is short by \cite[Lem\-ma~5.16]{AD}. 
Let~$K^*:=\R(\!(t^{\Gamma^*})\!)$ be the $H$-field closed under powers with constant field $\R$ and
  $H$-couple~$(\Gamma^*,\psi^*)$
over~$\R$ constructed in the previous subsection. Then by Remark~\ref{rem:restrict}, 
$K:=\R(\!(t^\Gamma)\!)$ is an $H$-subfield of $K^*$
with constant field $\R$ and asymptotic couple~$(\Gamma,\psi)$ and $K$ has small derivation. 
Since $K$ is short, it embeds into $L$
by Lem\-ma~\ref{lem:7.8}.
\end{proof}

\noindent
Lemma~\ref{shortsph} now yields Corollary~C from the introduction:

\begin{cor}
Let $M$ be a maximal Hardy field.
Then every short asymptotic couple  of Hardy type with small derivation is isomorphic to the
asymptotic couple of a spherically complete $H$-subfield of $M$ containing $\R$.
Likewise with ``maximal smooth'' $($respectively,  ``maximal analytic''$)$ in place of ``maximal''.
\end{cor}

\printendnotes

\newlength\templinewidth
\setlength{\templinewidth}{\textwidth}
\addtolength{\templinewidth}{-2.25em}

\patchcmd{\thebibliography}{\list}{\printremarkbeforebib\list}{}{}

\let\oldaddcontentsline\addcontentsline
\renewcommand{\addcontentsline}[3]{\oldaddcontentsline{toc}{section}{References}}

\def\printremarkbeforebib{\bigskip\hskip1em The citation [ADH] refers to the book \\

\hskip1em\parbox{\templinewidth}{
{\sc M. Aschenbrenner, L. van den Dries, J. van der Hoeven,}
\textit{Asymptotic Differential Algebra and Model Theory of Transseries,} Annals of Mathematics Studies, vol.~195, Princeton University Press, Princeton, NJ, 2017.
}

\bigskip

}

\bibliographystyle{amsplain}

\end{document}